\documentclass[11pt]{article}

\usepackage{amsfonts}
\usepackage{amsmath}

\newcommand{\beq}{\begin{equation}}
\newcommand{\eeq}{\end{equation}}
\newcommand{\bea}{\begin{eqnarray}}
\newcommand{\eea}{\end{eqnarray}}
\newcommand{\beas}{\begin{eqnarray*}}
\newcommand{\eeas}{\end{eqnarray*}}

\newcommand{\e}{\varepsilon}

\newtheorem{theorem}{Theorem}[section]

\newtheorem{proposition}[theorem]{Proposition}
\newtheorem{corollary}[theorem]{Corollary}
\newtheorem{lemma}[theorem]{Lemma}
\newtheorem{remark}[theorem]{Remark}
\newtheorem{example}[theorem]{Example}
\newtheorem{examples}[theorem]{Examples}
\newtheorem{foo}[theorem]{Remarks}

\newenvironment{proof}{\addvspace{\medskipamount}\par\noindent{\it
Proof}.}
{\unskip\nobreak\hfill$\Box$\par\addvspace{\medskipamount}}

\newcommand{\R}{\mathbb{R}} 
 
\newcommand{\Z}{\mathbb{Z}} 
\newcommand{\C}{\mathbb{C}} 
\newcommand{\HH}{\mathbb{H}}  
\newcommand{\G}{\mathbb G}
\newcommand{\ch}{\cosh}
\newcommand{\sh}{\sinh}
\newcommand{\SU}{\mathbf{SU}(2)}
\newcommand{\SL}{\mathbf{SL}(2,\mathbb{R})}
\newcommand{\tSL}{\widetilde{\SL}}
\newcommand{\tX}{\tilde{X}}
\newcommand{\tY}{\tilde{Y}}
\newcommand{\tZ}{\tilde{Z}}
\newcommand{\tp}{\tilde{p}}
\newcommand{\tP}{\tilde{P}}
\newcommand{\tL}{\tilde{L}}
\newcommand{\tmu}{\tilde{\mu}}
\newcommand{\td}{\tilde{d}}
\newcommand{\tGa}{\tilde {\Gamma}}
\newcommand{\hX}{X_\HH}
\newcommand{\hY}{Y_\HH}
\newcommand{\hZ}{Z_\HH}

\newcommand{\hP}{P^\HH}
\newcommand{\hL}{L_\HH}
\newcommand{\hGa}{\Gamma^\HH}

\newcommand{\tA}{\tilde{A}}
\newcommand{\hA}{A_\HH}

\newcommand{\arch}{\mathrm{arch}}

\title{ The subelliptic heat kernels on  $\mathbf{SL}(2,\mathbb{R})$ and on its universal covering $\tSL$:  integral representations and some functional inequalities}

\author{Michel Bonnefont \footnote{michel.bonnefont@math.u-bordeaux1.fr} \\
{\small Institut de Math\'ematiques de Bordeaux} \\
{\small Universit\'e de Bordeaux 1} \\
{\small CNRS UMR 5251} \\
}

\begin{document}

\maketitle

\begin{abstract}
In this paper, we study a subelliptic heat kernel on the Lie group $\mathbf{SL}(2,\mathbb{R})$ and on its universal covering $\tSL$. The subelliptic structure on $\mathbf{SL}(2,\mathbb{R})$ comes from the fibration $SO(2)\rightarrow \mathbf{SL}(2,\mathbb{R}) \rightarrow H^2$ and it can be lifted to $\tSL$. First, we derive an integral representation for these  heat kernels. These expressions allows
 us to obtain some asymptotics in small times of the heat kernels and give us a way to compute the subriemannian distance. Then, we establish some gradient estimates and some functional inequalities like a Li-Yau type estimate and a reverse Poincar\'e inequality that are valid for both heat kernels. 
\end{abstract}

\tableofcontents

\section{Introduction}
The goal of this work is to study a particular subelliptic structure on the Lie group $\SL$ and on its universal covering $\tSL$. It will correspond to study heat kernels of  operators which can be written as a sum of squares of vector fields and which satisfy the so-called H\"ormander condition. For a general account on this subject one can consult the two monographs \cite{VSC} and \cite{Robinson}.

\
Here for $\SL$, the subelliptic structure is coming from the fibration: $SO(2)\rightarrow \SL \rightarrow H^2$ where $H^2$ is the 2-dimensional hyperbolic space.
In this fibration the metric on $\SL$  is the one inherited from the Killing form and is Lorentzian of signature (2,1). The restriction of this metric to the horizontal distribution is of signature (2,0) and  gives the subellipitic structure. This is described more precisely in Section \ref{II}  (see  also \cite{montgomery}). Moreover, as seen in Section \ref{II}, this subelliptic structure can easily be lifted to $\tSL$.

\
This space $\tSL$ can be proposed as the model space of a negatively curved 3-dimensional subriemannian manifold. To be more precise, it should be proposed as the model space of a 3-dimensional CR-manifold with vanishing pseudo-Hermitean torsion (Sasaki manifolds) and with constant negative curvature (see \cite{Dragomir} for an account on CR-manifolds). The present work is coming after some analogous studies on the Heisenberg group \cite{bakry-baudoin-bonnefont-chafai} and on the canonical subelliptic Lie group $\SU$ \cite{Baudoin-Bonnefont}. The subelliptic structure on this last group  is very similar to the one studied here. The Heisenberg group plays the role of the Euclidean space in this geometry whereas 
$\SU$ stands for the positively curved model space. As we will see it in the sequel, these three structures share a lot of results in common.

\

The precursor work before the study of the Heisenberg group is due to L\'evy who studied the area swept out by a two dimensional Brownian motion \cite{Levy}. After that, the study of the heat kernel on the Heisenberg group began really with Hulanicki \cite{hula} and Gaveau \cite{Gaveau}. In \cite{Gaveau}, Gaveau established an integral representation of the heat kernel which is now known as the Gaveau formula. This enabled him to obtain some asymptotics in small time of the heat kernel and even,  more recently, with Beals and Greiner, to obtain some optimal bounds for the heat kernel (see \cite{Gaveau2} and also \cite{hueber-muller}, \cite{HQLi}). Recently, the focus was on obtaining  functional inequalities and gradient estimates on this group. For example, a subcommutation between the gradient and the semi-group were derived in \cite{Driver-Melcher}, \cite{HQLi} and \cite{bakry-baudoin-bonnefont-chafai}.

\

In \cite{Baudoin-Bonnefont}, a study of the subelliptic heat kernel on $\SU$ was done. Our study here is very closed to this one since the structures are very similar.  Using the isomorphism between $\SU$ and the 3-sphere $S^3$, the autors managed to obtain an integral representation of the heat kernel on $\SU$. This representation is based on the relations between the sublaplacian and the classical Laplace-Beltrami operator on $S^3$. Therefore the integral representation makes appear the classical heat kernel on $S^3$. Here on $\tSL$, 
it is still possible to obtain an integral representation  of the heat kernel in which the classical heat kernel on $H^3$ appears. 
In fact, this is linked with  the relation between the sublaplacian and the Casimir operator (see remark \ref{geometric_interp}). The heat kernel on $\SL$ is then just obtained by wrapping the one of $\tSL$. 
\

With these formulas, we are able to obtain asymptotics in small time of the heat kernels, together with some asymptotics in large time of the heat kernels on the diagonal. Moreover, we can  derive some ultracontractive bounds, a way to compute the subriemannian distances and the convergence of these diffusions towards the one on the Heisenberg group. Note that in small times, the two heat kernel have the same behaviour, since then  the leading term in the heat kernel of $\SL$ is exaclty the heat kernel on $\tSL$. 
We are also able to derive some gradient estimates and  functional inequalities like  Li-Yau type estimates,  reverse Poincar\'e inequalities and some isoperimetric inequalities. They are derived only through a local study of the subelliptic structure, thus they read exactly the same on both $\SL$ and $\tSL$.  These inequalities are also valid on $\HH$ and $\SU$ (see \cite{bakry-baudoin-bonnefont-chafai}, \cite{bbbq} and \cite{Baudoin-Bonnefont}).  

\

This paper is divided into three parts.
In the first one, we recall some basics facts about the Lie group $\SL$, we describe precisely the subelliptic structure we consider and how we lift it to $\tSL$. We also introduce some cylindrical coordinates that we will use in the sequel.
In the second one, we derive the integral representation of the heat kernel and give its consequences. In the last one, we establish some gradient estimates and some functional inequalities for the heat kernels.

\section{Preliminaries on $\SL$ and on $\tSL$}\label{II}
In this section, we describe the subelliptic structures that we consider on both $\SL$ and $\tSL$.
We concentrate first on the Lie group $\SL$ since  it can be represented as a subgroup of $\mathbf{GL}(2,\R)$.
We will explain then, at the end of the section, how to lift the subellitpic structure to $\tSL$.

\
The Lie group $\mathbf{SL} (2,\R)$ is the group of
$2 \times 2$, real matrices of determinant $1$. Its
Lie algebra $\mathfrak{sl} (2,\R)$ consists of $2 \times 2$ matrices of trace $0$.
A basis of $\mathfrak{sl} (2,\R)$ is formed by the matrices:
 \[
X=\left(
\begin{array}{cc}
~1~ & ~0~ \\
~0~ & -1~
\end{array}
\right)
,\text{ }Y=\left(
\begin{array}{cc}
~0~ & ~1~ \\
~1~ & ~0~
\end{array}
\right),
\text{ }Z=\left(
\begin{array}{cc}
~0~ & ~1~ \\
-1~& ~0~
\end{array}
\right)
,
\]
for which the following relations hold
\begin{align}\label{Liestructure}
[X,Y]=2Z, \quad [X,Z]=2Y, \quad [Y,Z]=-2X.
\end{align}

We associate to these matrices the left-invariant vector fields they generate, which we still denote by the same letters. For example, for a smooth function $f$ on $\SL$ and $g \in \SL$: 
$$X(f)(g)= \lim_{t\rightarrow 0 } \frac{1} {t} \left( f(g . \exp(t X) ) - f(g)\right). 
$$   
Below we will see the expressions of these vector fields in some coordinates.

Now we consider on this Lie group the left-invariant, second order differential operator
$$L= X^{2}  + Y^{2}$$
as well as the heat semigroup
$$ P_t=e^{t L}.$$

Due to H\"ormander's theorem and 
 the structure of the Lie algebra (\ref{Liestructure}), the operator $L$  
  is subelliptic. Therefore the heat semi-group $(P_t)_{t>0}$ admits a smooth density with respect to its invariant measure.
  
   The operator is subelliptic but not elliptic so that the associated geometry is not Riemannian 
but only subriemannian. The notion of distance associated to the operator $L$ is given by 
\[
d(g_1,g_2)=\sup_{f \in \mathcal{C}} \{ \mid f(g_1) -f(g_2) \mid \}
\]

where $\mathcal{C}$ is the set of smooth maps $\SL \rightarrow \mathbb{R}$ that satisfy
 $(Xf)^2+(Yf)^2  \le 1$. Via Chow's theorem, this distance can also be defined as the minimal length
of horizontal curves joining two given points (see Chapter 3 of \cite{baudoin}). This distance is called the Carnot-Carath\'eodory distance.

Let us now describe more precisely the links between  this subelliptic structure and the fibration: $SO(2)\rightarrow \SL \rightarrow H^2$. 
 In the above fibration, we let act $\SL$ on the Poincar\'e hyperbolic upper plane  
 by the homographies:
$$ R_{M_{a,b,c,d}}:z \in H^2 \to \frac{az+b}{cz+d} \in H^2
\textrm{ for } M_{a,b,c,d}=
\left( \begin{array}{c c}
a&b\\ c&d
\end{array}\right)
\textrm{ a point of }\SL.$$
This is a transitive action by isometries of $H^2$.
Note that the matrices $M$ and $-M$ induced the same homography.
To obtain the fibration explicitly, we look a the image of a particular point of $H^2$ for example the point $i$. We obtain then the smooth map:
$$ \phi: M_{a,b,c,d} \in \SL \to \frac{ai+b}{ci+d} = \frac{bd +ac}{c^2+d^2} + i \frac{1}{c^2+d^2}\in H^2.
$$ 
Now easy computations show that the differential of $\phi$ sends the left invariant vector fields $X$ and $Y$ on an orthognonal basis of $H^2$ and that $d\phi. Z=0$.
\

For a better understanding of our subelliptic model, recall that the Killing form on $\SL$ is  the bilinear form given by $k(U,V)= \textrm{trace} ( \textrm{ad}(U) \textrm{ad}(V))$ for $U,V \in \mathfrak{sl} (2,\R)$. In our basis $(X,Y,Z)$ of $\mathfrak{sl} (2,\R)$, it is given by the matrix:

$$
\left(
\begin{array}{c c c}
8 &0&0\\
0 &8&0\\
0&0& -8\\
\end{array}
\right).
$$

The Lorentzian metric associated to the Killing form is therefore $8(dX^2+dY^2-dZ^2)$ and the above fibration is then a pseudo-Riemannian 
submersion from $\SL$ with this pseudo-Riemanian metric over the hyperbolic space of dimension 2.

\begin{remark}\label{cartan_decomposition} This is also related to the following Cartan decomposition of $\mathfrak{sl}(2,\R)$.
Let $\theta$ be the linear map defined by $\theta (X)=-X, \theta (Y)=-Y, \theta(Z)=Z$. It is an involution of $\mathfrak{sl}(2,\R)$, a Lie algebra automorphism and is such that the bilinear form
$B_\theta(U,V)= - k(U,\theta V)$ is definite positive and therefore is a Cartan involution of $\mathfrak{sl}(2,\R)$.
Thus a Cartan decomposition of $\mathfrak{sl}(2,\R)$ is given by 
$$\mathfrak{sl}(2,\R)= \mathfrak{f} \oplus \mathfrak{p}$$
where $\mathfrak{f}=\textrm{Vect} (Z)$ the eigenspace associated to the eigenvalue $1$ of $\theta$ and $\mathfrak{p}=\textrm{Vect} (X,Y)$ the eigenspace associated to the eigenvalue $-1$ of $\theta$. Of course the following holds:
$$
[\mathfrak{f},\mathfrak{f}]\subseteq \mathfrak{f}, \; [\mathfrak{f},\mathfrak{p}]\subseteq \mathfrak{p} \textrm{ and } [\mathfrak{p},\mathfrak{p}]\subseteq \mathfrak{f}.
$$
\end{remark} 

\

Now we come back to the study of the operator $L$, we introduce the cylindrical coordinates:
\[
(r,\theta, z) \rightarrow \exp \left(r \cos \theta X +r \sin \theta Y \right) \exp (z Z)
\]
\[= \left(
\begin{array}{cc}
 \ch(r) \cos (z) + \sh(r)\cos (\theta +z) & \ch(r) \sin (z) + \sh(r)\sin (\theta +z)\\
-\ch(r) \sin (z) + \sh(r)\sin (\theta +z)& \ch(r) \cos (z)  - \sh(r)\cos (\theta +z)
\end{array}\right),
\]
with 
\[
 r >0, \text{ }\theta \in [0,2\pi], \text{ }z\in [-\pi,\pi].
\]

These coordinates are the equivalent in our context  of the ones  which were used  in \cite{Baudoin-Bonnefont} to study a similar subelliptic operator on the Lie group $\SU$.

Simple but tedious computations show that in these coordinates, the left-regular representation sends the matrices
 $X$, $Y$ and $Z$ to the
left-invariant vector fields: 
\[
X=\cos (\theta + 2 z) \frac{\partial}{\partial r}-\sin (\theta +2z) \left( \tanh r \frac{\partial}{\partial z}
+\left(\frac{1}{\tanh r}-\tanh r\right)  \frac{\partial}{\partial \theta}\right),
\]
\[
Y=\sin (\theta +2 z) \frac{\partial}{\partial r} +\cos (\theta + 2z) \left( \tanh r \frac{\partial}{\partial z}
+\left(\frac{1}{\tanh r}- \tanh r)\right)  \frac{\partial}{\partial \theta} \right),
\]
\[
Z=\frac{\partial}{\partial z}.
\]

We therefore obtain
\begin{align*}
L & =X^2+Y^2 \\
 & =\frac{\partial^2}{\partial r^2}+  2 \coth 2r \frac{\partial}{\partial r}
  +\tanh^2 r \frac{\partial^2}{\partial z^2} 
 + \frac{4}{\sinh^2 2r} \frac{\partial^2}{\partial \theta^2}
+2(1-\tanh^2 r) \frac{\partial^2}{\partial \theta \partial z}.
\end{align*}

The invariant and, in fact, also symmetric measure for $L$ is then given (up to a constant) by
\[
d\mu= \frac{\sinh 2r}{2} dr d\theta dz.
\]
The choice of the constant is made to obtain a good convergence towards the Lebesgue measure of $\R^3$ which is the invariant measure for the Heisenberg group (see section \ref{convH}).
Recall the group $\SL$ is unimodular and  note that the invariant measure $\mu$ coincides with the bi-invariant Haar measure of the group.
Note also that $L$ commutes with  $\frac{\partial}{\partial \theta}$ and with  $\frac{\partial}{\partial z}$.
From the commutation with $\frac{\partial}{\partial \theta}$, we deduce that the heat kernel (issued from the identity) only depends on $(r,z)$. It will then be denoted by $p_t(r,z)$. 

\

\

Let us now introduce the universal covering $\tSL$ of $\SL$ and the subelliptic geometry we consider on it. First note that $\SL$ is homeomorphic to $\R^2 \times S^1$ and  therefore  $\tSL$ is homeomorphic to $\R^3$. With our cylindrical coordinates, it can be represented by $(r,\theta, z) \in \R_+^*\times [0,2\pi] \times \R$ and the projection from $\tSL$ to $\SL$ is just obtained by the quotient $\R/2\pi \Z$ on the variable $z$. Of course  the vector fields $X$,$Y$ and $Z$ on $\SL$  can be lifted into vector fields $\tX$, $\tY$ and $\tZ$ on $\tSL$.  They are given by the sames formulas, but defined for all $(r,\theta, z) \in \R_+^*\times [0,2\pi] \times \R$: 
\[
\tX=\cos (\theta + 2 z) \frac{\partial}{\partial r}-\sin (\theta +2z) \left( \tanh r \frac{\partial}{\partial z}
+\left(\frac{1}{\tanh r}-\tanh r\right)  \frac{\partial}{\partial \theta}\right),
\]
\[
\tY=\sin (\theta +2 z) \frac{\partial}{\partial r} +\cos (\theta + 2z) \left( \tanh r \frac{\partial}{\partial z}
+\left(\frac{1}{\tanh r}- \tanh r)\right)  \frac{\partial}{\partial \theta} \right),
\]
\[
\tZ=\frac{\partial}{\partial z}.
\]

The vector fields $\tX, \tY$ and $\tZ$ satisfy obviously the same bracket relations as the vector fields $X,Y$ and $Z$ on $\SL$. 
We now consider the subelliptic operator $\tL=\tX+\tY$ on $\tSL$ and the associated semigroup $\tP_t=e^{t \tL}$. They share obviously the same properties as $L$ and $P_t$. We denote by $\td$ the distance associated to the operator $\tL$, $\tp_t(r,z)$ the associated heat kernel and $d\tmu= \frac{\sh 2r}{r} dr d\theta dz$ the invariant and symmetric measure for $\tL$.

\section{The subelliptic heat kernels on $\SL$ and $\tSL$}
 
\subsection{Integral representation of the kernel} 
 Let us consider the second order differential operator on the interval $[1,\infty)$
 $$\mathcal J = (x^2-1) \frac{d^2}{dx^2} +3 x \frac{d}{dx}
 $$
 with invariant and symmetric measure $(x^2-1)^{1/2}$.
 It is well known (see \cite{TaylorPDE2}) that  the heat kernel  $s_t$ associated to $\mathcal J$ issued from $1$  has the following expression for $x\geq 1$:
 \begin{equation}
s_t(x)=  \frac{e^{-t}}{\sqrt{4\pi} t^{3/2}} \left( \frac{ \arch x}{\sqrt{x^2-1}} \right) e^{-\frac{(\arch x)^2}{4t}}.
\end{equation}
 That is, for $f$ a smooth function $[1,\infty) \rightarrow \R$,  
 \[
(e^{t \mathcal{J}}f)(1)= \int_{1}^\infty s_t (x) f(x)(x^2-1)^{1/2} dx.
\]
It is clear the function $x\rightarrow (\arch x)^2$ admits an holomorphic extension to  $\C-\{]\-\infty,1] \}$; but in fact,
using Schwarz symmetry principle, we can  see that this extension is holomorphic on  $\C-\{]\-\infty,-1] \}$. Therefore this is the same for its derivative: $x\rightarrow \frac{ \arch x}{\sqrt{x^2-1}}$. So the heat kernel $s_t$ itself admits an holomorphic extension to $\C-\{]\-\infty,-1] \}$.
 By setting  $x=\ch r, r\geq 0$, we have
 \begin{equation}\label{st}
 s_t(\ch r)= \frac{e^{-t}}{\sqrt{4\pi} t^{3/2}} \left( \frac{r}{\sh r} \right) e^{-\frac{r^2}{4t}}.
 \end{equation}
This heat kernel corresponds in fact to the one on the 3-dimensional hyperbolic space. The difference of the factor $4\pi$ with the usual expression is coming from the fact that the invariant measure is $\sh^2r$ instead of usually $4\pi \sh^2 r$ which is the area of the sphere of radius $r$ in the 3-dimensional hyperbolic space.

 Now easy calculations give us that $s_t$ satisfies the following expressions:
 \begin{equation}\label{d1}
 \partial_t s_t( \ch r \cos z)=\Delta_1 (s_t(\ch r \cos z))   
 \end{equation}
 where $\Delta_1=\partial_{r,r} ^2 + 2 \coth 2r \partial_r + (\tanh^2 r -1) \partial_{z,z} ^2$ 
 and 
 \begin{equation}\label{d2}
  \partial_t s_t( \ch r \ch y)= \Delta_2 (s_t(\ch r \ch y))
 \end{equation}
 where $\Delta_2=\partial_{r,r} ^2 + 2 \coth 2r \partial_r + (1-\tanh^2 r) \partial_{y,y} ^2$.
 
 $\Delta_1$ and $\Delta_2$ are two self-adjoint operators respectively on  $(0,\infty)\times [-\pi,\pi]$ and on $(0,\infty)\times (0,\infty)$ with respective symmetric measure  $\frac{\sh 2r }{2} dr dz$ and $\frac{\sh 2r }{2} dr dy$. $\Delta_1$ is a hyperbolic operator whereas $\Delta_2$ is an elliptic operator.
For a geometric interpretation of $\Delta_1$, see remark \ref{geometric_interp}.

\begin{lemma}\label{hkd2}
If $f$ is a smooth function  $[0,\infty) \times [0,\infty) \  \rightarrow \mathbb{R}$,
then for $t \ge 0$,
\[
(e^{t \Delta_2}f)(0,0)
=\int_{r> 0} \int_{y>0} 
s_t (\cosh r \cosh y ) f(r,y)\frac{\sinh 2r}{2}  dr dy
\]
\end{lemma}

\begin{proof}
Indeed we saw that $s_t$ satisfies the equation:
 $$\partial_t s_t( \ch r \ch y)= \Delta_2 (s_t(\ch r \ch y)).$$
 Now we must check the initial condition. We have to show that for a smooth function $f$: $[0,\infty) \times [0,\infty) \  \rightarrow \mathbb{R}$:
 $$\int_{r> 0}  \int_{y>0} 
s_t (\cosh r \cosh y ) f(r,y) \frac{\sinh 2r}{2} dr dy \to f(0,0) \textrm{ when } t\to 0.$$

Since we will make the following change of variables:
$$\left\{\begin{array}{ccc}
u&=& \cosh r\cosh y\\
v&=&\cosh r \sinh y
\end{array}\right.
$$
we take the function $f$ of the form $f(r,z)=g(\cosh r \cosh z)h(\cosh r \sinh z)$. The new domain is $D=\{(u,v),u\geq 1,v \geq 0,u^2-v^2\geq 1 \}$ and the Jacobian determinant is $\frac{1}{2}\sinh 2r$. So
\begin{eqnarray*}
& &\int_{r>0}\int_{y>0} s_t(\cosh r \cosh y)g(\cosh r \cosh y) h(\cosh r \sinh y)\frac{\sinh 2r}{2} dr dy\\
&=& \int \int_D s_t(u)g(u)h(v)du dv\\
&= &\int_{u\geq 1} \left(\int_0^{(u^ 2-1)^ {1/2}} h(v) dv\right) s_t(u)g(u) du
\end{eqnarray*}

We may rewrite it as
$$\int_{u\geq 1} s_t(u)l(u) (u^2-1)^ {1/2}du 
$$
where $ l$ is the continuous fonction 
$$l(u)=g(u) \left( \frac{\int_0^{(u^2-1)^{1/2}} h(v) dv}{(u^2-1)^{1/2}}\right).
$$
Now, since $s_t$ is the heat kernel of a diffusion issued from $1$ with respect to the measure 
$(u^ 2-1)^{1/2} du$ and $l$ is continuous, the last quantity is converging towards $l(1)=g(1)h(0)=f(0,0)$ and the lemma is proved.
\end{proof}
  
With this, we can now derive an integral representation for the heat kernel on $\tSL$.

 \begin{proposition}\label{repint}
 The heat kernel on $\tSL$ is given  for $t>0, r>0 , z\in \R$ by
\begin{eqnarray*}
\tp_t (r,z)&=&\frac{1}{4\pi}\frac{1}{\sqrt{4\pi t}} \int_{-\infty}^{+\infty} e^{\frac{(y-iz)^2}{4t}} s_t (\ch r \cosh y) dy\\
         &=&\frac{e^{-t}}{(4\pi t)^2} \int_{-\infty}^{+\infty} e^{- \frac{\arch^2(\ch r \ch y) -(y-iz)^2 } {4t}}
            \frac {\arch(\ch r \ch y)} {\sqrt{\ch^2 r \ch^2 y - 1}}dy\\  
\end{eqnarray*}
 \end{proposition}
 \begin{proof}
 The second equality is just obtained by using the explicit value of $s_t$ and shows that the integral is well defined since it is absolutely convergent.
 Now let 
 $$q_t(r,z)=\frac{1}{4\pi} \frac{1}{\sqrt{4\pi t}} \int_{-\infty}^{+\infty} e^{\frac{(y-iz)^2}{4t}} s_t (\ch r \cosh y) dy.
 $$
 
 By using the fact that
\[
\frac{\partial}{\partial t} \left(\frac{ e^{\frac{(y-iz)^2}{4t}}}{\sqrt{4\pi t}}\right)=\frac{\partial^2}{\partial z^2} \left(\frac{ e^{\frac{(y-iz)^2}{4t}}}{\sqrt{4\pi t}}\right)=-\frac{\partial^2}{\partial y^2} \left(\frac{ e^{\frac{(y-iz)^2}{4t}}}{\sqrt{4\pi t}}\right)
\]
and
\[
\frac{\partial }{\partial t} (s_t(\ch r \ch y))  = \left(\partial_{r,r} ^2 + 2 \coth 2r \partial_r + (1-\tanh^2 r) \partial_{y,y} ^2 \right) (s_t(\ch r \ch y)),
\]
a double integration by parts with respect to the variable $y$ shows that 
$$
\frac{\partial q_t}{\partial t}=\frac{1}{4\pi} \frac{1}{\sqrt{4\pi t}} \int_{-\infty}^{+\infty} e^{\frac{(y-iz)^2}{4t}}  \Delta_3 (s_t (\ch r \cosh y)) dy
$$
where $\Delta_3= \partial_{r,r} ^2 + 2 \coth 2r \partial_r -\tanh^2 r \partial_{y,y} ^2$.

Now another double integration by parts in the variable $y$ shows us that
$$\frac{\partial}{\partial t} q_t(r,z)=  \tL q_t(r,z).
$$

Let us now check the initial condition. 
Let $f(r,z)=e^{i \lambda z} g(r)$ where $\lambda \in \R$ and $g$ is a smooth function.
We have
\begin{eqnarray*}
& & \int_{r>0} \int_{\theta=0}^{2\pi} \int_{z=-\infty}^\infty  q_t (r , z) f(r,z) \frac{\sinh 2r}{2} dr d\theta dz\\
 &=& \frac{1}{2} \int_{r>0} \int_{z=-\infty}^\infty \int_{y>0} \left( \frac{e^{-\frac{(z+iy)^2}{4t}} + e^{-\frac{(z-iy)^2}{4t}}}{\sqrt{4\pi t}} \right)
                       s_t(\ch r \ch y) g(r) e^{i\lambda z} \frac{\sinh 2r}{2} dr dz dy.
\end{eqnarray*}

 By changing the integration countour in the complex plane, we get:
$$
\int_{z=-\infty}^\infty \left( \frac{e^{-\frac{(z+iy)^2}{4t}}}{\sqrt{4\pi t}} \right) e^{i \lambda z}  dz = e^{\lambda y}\int_{z=-\infty}^\infty \left( \frac{e^{-\frac{z^2}{4t}}}{\sqrt{4\pi t}} \right) e^{i \lambda z}  dz
$$

We can do the same for the other term and eventually obtain, 
\begin{eqnarray*}
& & \int_{r>0} \int_{z=-\infty}^\infty  q_t (r , z) f(r,z) \frac{\sinh 2r}{2} dr dz\\
 &=&  \left(\int_{z=-\infty}^\infty  \frac{e^{-\frac{z^2}{4t}}}{\sqrt{4\pi t}}  e^{i \lambda z}  dz \right)  \int_{r>0} \int_{y>0} s_t(\ch r \ch y) g(r) \ch(\lambda y) \frac{\sinh 2r}{2} dr dy\\
& =& a(t) \; e^{t \Delta_2} (l)(0)
\end{eqnarray*}
where $l$ is the function $l(r,y)=  g(r) \ch(\lambda y)$ and 
 $a(t)=\int_{z=-\infty}^\infty  \frac{e^{-\frac{z^2}{4t}}}{\sqrt{4\pi t}}  e^{i \lambda z}  dz$. By  classical results on the heat kernel on $\R$, $a(t)$ tends to $e^{i\lambda 0}=1$ when $t$ goes to $0$. Similarly  $e^{t \Delta_2} (l)(0)$ tends to $l(0)$. Therefore this term is converging to $ g(0) = f(0,0)$ when $t$ goes to $0$. 
\end{proof}
 
An integral representation for the heat kernel on $\SL$ follows easily:

\begin{proposition}\label{repint_SL}
 The heat kernel on $\SL$ is given  for $t>0, r>0 , z\in [-\pi,\pi]$ by
\begin{eqnarray*}
p_t (r,z)&=&\sum_{k\in \Z} \tp_t(r,z+2k\pi)\\
         &=&\frac{e^{-t}}{(4\pi t)^2} \sum_{k\in \Z} \int_{-\infty}^{+\infty} e^{- \frac{\arch^2(\ch r \ch y) -(y-iz-i2k\pi)^2 } {4t}}
            \frac {\arch(\ch r \ch y)} {\sqrt{\ch^2 r \ch^2 y - 1}}dy\\  
\end{eqnarray*}
 \end{proposition}

\begin{proof}
The same proof as above works. Indeed, it is already clear that this kernel $p_t$ satisfies the heat equation $\partial_t p_t=L p_t$ since $L$ and $\tL$ write exactly the same. The initial condition is easily obtained by noticing that: 
$$
\sum_{k\in \Z} \int_{z=-\pi}^\pi \left( \frac{e^{-\frac{(z+2k\pi+iy)^2}{4t}}}{\sqrt{4\pi t}} \right) e^{i \lambda z}  dz =\int_{z=-\infty}^\infty \left( \frac{e^{-\frac{(z+iy)^2}{4t}}}{\sqrt{4\pi t}} \right) e^{i \lambda z}  dz.
$$
\end{proof}

\begin{remark}\label{geometric_interp}
It is not exactly the way we used to prove it, but we have the following geometric interpretation for the sublaplacian on $\SL$.:
$$L = \square + Z^2.
$$
where $\square$ stands for the Casimir operator $\square= X^2+Y^2-Z^2$ (see \cite{Taylor}). As $\square$ is in the center of the envelopping algebra of $\SL$ (and if fact generates it), we have also the following geometrical interpretation for the semigroup:
$$e^{t L}= e^{t Z^2} e^{t \square}.
$$ 
Note now that the operator $\Delta_1$ is nothing else than the radial part of the operator $\square$.
\end{remark}


\subsection{Asymptotics of the heat kernel in small time}
The goal of this section is to obtain the precise asymptotics of the heat kernels when $t \to 0$.

We will mainly  study the heat kernel on $\tSL$ since one can obtain more explicit formulas for the heat kernel. Moreover, all the asymptotics in small times for the heat kernel on $\SL$ are exactly the same as the ones for $\tSL$. Indeed, in small times the leading term in the sum is the  term for $k=0$;  which is exactly the heat kernel on $\tSL$. 

\
  
Of course, in large times, as we will see, the behaviours of the heat kernels on $\SL$ and $\tSL$ are different.

We start with the points of the form $(0,z), z\in \R$ that lie on the cut-locus of 0. For these points we have

$$\tp_t(0,z)=\frac{e^{-t}}{(4\pi t)^2} e^{-\frac{z^2}{4t}} \int_{-\infty}^{+\infty} e^{\frac{-iyz}{2t}} \frac{y}{\sh y} dy. 
$$
 
A computation of the integral is possible using residus calculus and gives the following: 
\begin{proposition} For $z\in \R$ and $t>0$,
$$\tp_t(0,z)=\frac{e^{-t}}{8 t^2} \frac{e^{-\frac{2 \pi |z| +  z^2}{4t}}} {\left(1+ e^{-\frac{ \pi |z|}{2t}} \right)^2} 
$$
therefore, for all $z\in \R$, when $t \rightarrow 0$, 
$$
\tp_t(0,z) \sim \frac{e^{-t}}{8 t^2} e^{-\frac{2\pi |z| +z^2}{4t}}.
$$
\end{proposition}

By continuity of the heat kernel we obtain the value on the diagonal.
\begin{proposition}\label{valeuren0} 
For $t>0$,
$$\tp_t(0,0)=\frac{e^{-t}}{32 t^2}. 
$$
\end{proposition}

Actually, for these points, a computation of the heat kernel on $\SL$ is also possible.

\begin{proposition} For $-\pi<z<\pi$ and $t>0$,
$$
p_t(0,z)=\frac{e^{-t}}{ 8 t^2} \sum_{k\in \Z} \exp \left(-\frac{(z+2k\pi)^2}{4t}\right) \frac{\exp \left(-\frac{|z+2k\pi|\pi}{2t}\right)}
{\left( 1+\exp\left(-\frac{|z+2k\pi|\pi}{2t}\right) \right)^2}
$$

and 

$$
p_t(0,0)=\frac{e^{-t}}{ 8 t^2} \sum_{k\in \Z} \exp \left(-\frac{k^2\pi^2}{t}\right) \frac{\exp \left(-\frac{|k|\pi^2}{t}\right)}
{\left( 1+\exp\left(-\frac{|k|\pi^2}{t}\right) \right)^2}.
$$
\end{proposition}  

\begin{remark}
With this last expression, it is possible to obtain the asymptotic of the heat kernel on $\SL$ on the diagonal in large time.
Indeed,
for $k\neq 0$, $$k^2 \leq |k|(|k|+1) \leq (|k|+1)^2,$$
 therefore
$$
 \frac{1}{4}\frac{e^{-t}}{ 8 t^2}  \sum_{k\in \Z, |k| \neq 1} \exp \left(-\frac{k^2\pi^2}{t}\right) \leq p_t(0,0) \leq \frac{e^{-t}}{ 8 t^2}  \sum_{k\in \Z} \exp \left(-\frac{k^2\pi^2}{t}\right).
$$
The well known identity follows from the Poisson summation formula:
$$
\sum_{k\in \Z} \exp \left(-\frac{k^2\pi^2}{t}\right) = \frac{\sqrt t}{\sqrt \pi} \sum_{k\in Z} \exp(-k^2 t)
$$  
which, when $t \to \infty$, is equivalent to $\frac{\sqrt t}{\sqrt \pi}$. Thus there exist two constants $c,C>0$ such that for all $t \geq 1$,
$$ \frac{c} { t^{\frac{3}{2}}} e^{-t} \leq p_t(0,0) \leq   \frac{C} { t^{\frac{3}{2}}} e^{-t}.$$

This kind of behaviour is already known in the litterature, see for example \cite{bougerol}. 
\end{remark}
Now we come back to $\tSL$ and turn to points of the form $(r,0)$. The next proposition give their asymptotics in small time for the heat kernel.
\begin{proposition}
For $r >0$, when $ t \to 0$,
\[
\tp_t (r,0) \sim   \frac{1}{(4\pi t)^{\frac{3}{2}}} \frac{r}{\sh r} \sqrt{\frac{1}{r \coth r -1} } e^{-\frac{r^2}{4t}}.
\]
\end{proposition}

\begin{proof}
We have for $r>0$
$$\tp_t(r,0)= \frac{e^{-t}}{(4\pi t)^2} \int_{-\infty}^{+\infty} e^{- \frac{\arch^2(\ch r \ch y) -y^2 } {4t}}
            \frac {\arch(\ch r \ch y)} {\sqrt{\ch^2 r \ch^2 y - 1}}dy
$$

We now analyze the  above integral in small times thanks to the Laplace method.

On $\R$, 
 the function
\[
f(y)=\arch (\ch r \cosh y))^2 -y^2
\]
has a unique minimum which is attained at $y=0$ and is equal to $r^2$, at this point:
$$
f''(0)=2(r \coth r -1). 
$$
The result follows by the Laplace method.
\end{proof} 
 
The previous proposition can be extended by the same method when $z \neq 0$. 
Let $r>0, z \in [-\pi, \pi]$ and consider the function
\[
f(y)=  (\mathrm{arch} (\ch r \cosh y))^2-(y-iz)^2,
\]
This function is well defined and holomorphic on the strip $|\mathrm{Im} (y) |  < \mathrm{arcos} \left(\frac{ - 1}{\ch r}\right)$ 
and it has for all $r>0, z\in \R$ a critical point at $i\theta (r,z)$ where $
\theta(r,z)$ is the unique solution in $(-\mathrm{arcos} \left(\frac{ - 1}{\ch r}\right),\mathrm{arcos} \left(\frac{ - 1}{\ch r}\right))$ to the equation:
\[
\theta (r,z)-z=\ch r \sin \theta (r,z)\frac{ \arch  (\ch r \cos \theta (r,z) ) }{\sqrt{\ch^2 r \cos^2 \theta (r,z) -1}}.
\]

Indeed the function $\theta \rightarrow  \ch r \sin \theta (r,z)\frac{ \arch  (\ch r \cos \theta (r,z) ) }{\sqrt{\ch^2 r \cos^2 \theta (r,z) -1}}$ is continuous, strictly increasing  from $-\infty$ to $\infty$ and with a derivative greater than $1$.

At the critical point, $f''(i\theta(r,z))$ is a positive and real number
$$f''(i\theta(r,z))= 2 \frac{\sh^2 r}{u(r,z)^2-1} \left[ \frac{u(r,z) \arch u(r,z)} {\sqrt {u(r,z)^2-1}}-1\right]$$

with $u(r,z)= \ch r \cos \theta(r,z)$ since $u>-1$.

We may observe that  $z$ and $\theta(r,z)$ have opposite signs.

By the same method than in the previous proposition, we obtain:
\begin{proposition}
Let $r>0, z \in \R$. When $t \to 0$,

$$\tp_t(r,z) \sim \frac{1}{\sinh r} \frac{\mathrm{arccosh } u(r,z)}{\sqrt{   \frac{u(r,z) \mathrm{arcosh } u(r,z)}{\sqrt{u^2(r,z)-1} }-1}} \frac{e^{-\frac{(\theta (r,z)-z)^2 \tanh^2 r}{4t \sin^2 \theta (r,z)}}}{(4\pi t)^{\frac{3}{2}}}
$$
with $u(r,z)= \ch r \cos \theta(r,z)$.
\end{proposition}


 
Of course, as said before, the heat kernel on $\SL$ share exactly the same asymptotics in small time.

\begin{remark}
According to L\'eandre results \cite{leandre1} and \cite{leandre2} (see also \cite{hino}), the previous asymptotics give a way to compute the subriemannian distance from 0 to the point $(r,\theta,z) \in \tSL$ by computing $\lim_{t \to 0} -4t \ln p_t (r,z)$. This distance does not depend on the variable $\theta$ and shall be denoted by $d(r,z)$.
\begin{itemize}
\item For $z \in \R$, 
\[
\td^2 (0,z)=2\pi \mid z \mid  + z^2 .
\]
\item For $r>0$, 
\[\td^2(r,0)=r^2. 
\]
\item For $z \in \R$, $r>0$,  
\[
\td^2(r,z)=\frac{(\theta(r,z)-z)^2 \tanh^2 r}{ \sin^2 \theta (r,z)}.
\]
\end{itemize}

Of course, the same result is true for the distance $d(r,z)$ on $\SL$ for $r>0, z\in [-\pi,\pi]$ where $d(r,z)$ is defined in the same way as above. In fact, the two distances $d$ and $\td$ coïncide for $r>0$ and $z\in [-\pi,\pi]$. 
\end{remark} 

From this remark we can get some estimates of the distance:
\begin{proposition}\label{estimates_distance}  
There exist two constants $c, C >0$ such that for all $r>0$ and $z\in[-\pi,\pi]$:
\[
c \max(r^2, |z|,|z|^2) \leq \td^2(r,z) \leq C \max(r^2,|z|, z^2).
\]  
\end{proposition}
\begin{proof}
For the right inequality, as in our coordinates on the group $\tSL$, $(r,0,0)*(0,0,z) = (r,0,z)$, we obtain by using the left invariance of the distance:
$\td(r,z) \leq \td(r,0) + \td(0,z)$. By combining it with the previous result, for all $r>0$ and $z\in \R]$, we get:
\[
\td^2(r,z) \leq C \max(r^2,|z|, z^2) 
\]
where $C$ is a positive constant.

\

Let us turn to the left inequality. 
Since $(r,0,z)*(0,0, -z) = (r,0,0)$, then $\td(r,0)- \td(0,z) \leq \td(r,z)$ and so 
 the result is true in the region where $r^2 \geq A  \max(|z|,|z|^2)$ with $A$ big enough.
  
Similarly, since $(r,\pi,0)*(r,0,z) = (0,0,z)$ then $\td(0,z) - \td(r,0) \leq \td(r,z)$ and the result is true in the region where $\max(|z|,z^2) \geq B r^2$ with $B$ big enough.  
 Now, consider the region $\{(r,z), \frac{1}{A} r^2\leq \max( |z|,z^2) \leq  B r^2\}$. 
  Recall that $\theta(r,z)$ and $z$ have opposite signs. Therefore 
$$\frac{(\theta(r,z)-z)^2}{ \sin^2 \theta (r,z)}\geq 1+ 2 |z|+ z^2 \geq 1 + \max(|z|,z^2).$$
Moreover $\tanh^2 r \geq c' \min(1, r^2)$.
So on the above domain, the expression of the distance gives:
$$ \td^2(r,z) \geq c'\min( 1, r^2)  (1 + \max(|z|,z^2)).
$$  
On the considered domain, the function on the right side behaves like $\max(r^2, |z| ,z^2)$ and gives the result.
\end{proof}

As a consequence, on $\SL$,  there exist two constants $c, C >0$ such that for all $r>0$ and $z\in[-\pi,\pi]$:
\[
c \max(r^2, |z|) \leq \td^2(r,z) \leq C \max(r^2,|z|).
\]

\

The proposition \ref{valeuren0} gives that the heat kernel  on $\tSL$ satisfies the following ultracontractivity bound:
\begin{equation}\label{ultra}
\tp_t(0,0)=||\tp_t||_\infty \leq \frac{e^{-t}}{32 t^2}. 
\end{equation}

Now by using well known results from Davies (see \cite{davies} or \cite{VSC}), this leads to the following general gaussian upper estimate (where we do not take into account the exponential decay): 
\beq \label{upper_estimate1}
\tp_t(r,z)\leq \frac{C_\eta} {t^2} \exp \left( - \frac{d^2(r,z)}{4(1+\eta) t}\right)
\eeq
where $C_\eta$ is a constant which depends on $\eta>0$.

Then by combining (\ref{ultra}) and (\ref{upper_estimate1}), one gets the better estimate:
\begin{proposition}\label{upper_estimate}
For all $\e>0$, there exist two positive constants $C_\e$ and $\delta_\e$ such that
\[
\tp_t(r,z)\leq C_\e \frac{e^{-  \delta_\e t}}{t^2} \exp \left( - \frac{d^2(r,z)}{4(1+\e) t}\right).
\]
\end{proposition} 


\subsection{From $\SL$ and $\tSL$ to Heisenberg}\label{convH} 
Let us first recall some basic properties of the three-dimensional Heisenberg group 
(see by e.g. \cite{baudoin}, \cite{bakry-baudoin-bonnefont-chafai} and the references therein): $\mathbb{H}$ can
be represented as $\mathbb{R}^3$ endowed with the polynomial group law:
\[
(x_1,y_1,z_1) (x_2,y_2,z_2)=(x_1+x_2,y_1+y_2,z_1+z_2+x_1y_2-x_2y_1).
\]
The left invariant vector fields read in cylindrical coordinates ($x=r \cos \theta, y=r\sin \theta$):
\begin{align}\label{Xheisenberg}
\hX= \cos \theta \frac{\partial }{\partial r} -\frac{\sin
\theta}{r}\frac{\partial }{\partial \theta}- r \sin
\theta \frac{\partial }{\partial z}
\end{align}
\begin{align}\label{Yheisenberg}
\hY= \sin \theta \frac{\partial }{\partial r} +\frac{\cos
\theta}{r} \frac{\partial }{\partial \theta} + r \cos
\theta \frac{\partial }{\partial z}
\end{align}
\begin{align}\label{Zheisenberg}
\hZ=\frac{\partial }{\partial z}.
\end{align}
And the following equalities hold
\[
\lbrack \hX,\hY]=2\hZ,\text{ }[\hX,\hZ]=[\hY,\hZ]=0.
\]
We denote
\[
\hL=\hX^2+\hY^2.
\]
and 
\[
 \hGa (f,f)= (\hX f)^2 + (\hY f)^2.
\]
Due to Gaveau's formula (see \cite{hula}, \cite{Gaveau}), with respect to the Lebesgue measure $r dr d\theta dz$ the heat kernel associated 
to the semigroup $(\hP_t)_{t\geq0}=(e^{t\hL})_{t\ge0}$ writes
\begin{equation}\label{gaveau}
h_t(r,z)= \frac{1}{16 \pi^2} \int_{-\infty}^{+\infty} e^{\frac{i \lambda z}{2}} \frac{\lambda}{\sinh \lambda t}
e^{-\frac{r^2}{4}\lambda  \text{cotanh}  \lambda  t }d \lambda.
\end{equation}


From a metric point of view it is known that the Heisenberg group is the tangent cone in the Gromov-Hausdorff sense. This means that  balls of radius $R$ for a dilating distance on $\SL$ or $\tSL$ are getting closer and closer in a certain sense of the balls of the same radius $R$ of the Heisenberg group. For a precise statement of it, see Mitchell theorem \cite{Mit} (see also \cite{baudoin}). Here we will see some more precise results.

First, in our setting, the dilation of $\SL$ or $\tSL$ towards the Heisenberg group can be seen at the level of differential operators.
As before, since this is about the behaviour in small times, it works exaclty the same for both $\tSL$ and $\SL$. 

Through the map
\begin{eqnarray*}
\tSL &\rightarrow &\mathbb{H} \\
\exp(r(\cos \theta \tX +\sin \theta \tY)) \exp {z\tZ} & \rightarrow & (r,\theta,z)
\end{eqnarray*}
we can see the vector fields $\tX$, $\tY$ and $\tZ$ of $\tSL$ as first order differential operators acting 
on smooth functions on the Heisenberg group.

Let us now denote by $D$ the dilation vector field on $\mathbb{H}$ given in cylindrical coordinates by
\[
D=r \frac{\partial}{\partial r} +2z\frac{\partial}{\partial z}
\] 
For $c\ge 1$ and $\tA=\tX,\tY,\tZ$ we denote by $\tA^c$ the dilated vector field:
$$
\tA^c=\frac{1}{\sqrt{c}} e^{-\frac{1}{2} \ln c D}\,  \tA \, e^{ \frac{1}{2} \ln c D} ,
$$

In the cylindrical coordinates of the Heisenberg group, we have

\[
\tX^c=\cos (\theta +\frac{2z}{c}) \frac{\partial}{\partial r}-\sin (\theta + \frac{2z}{c} )
 \left( \sqrt{c}\tanh \frac{r}{\sqrt{c}} \frac{\partial}{\partial z}
+\left(\frac{1}{\sqrt{c}\tanh \frac{r}{\sqrt{c}}} - \frac{\tanh \frac{r}{\sqrt{c}}}{\sqrt{c} }\right) 
 \frac{\partial}{\partial \theta}\right),
\]

\[
\tY^c=\sin(\theta + \frac{2z}{c}) \frac{\partial}{\partial r}+\cos (\theta +\frac{2z}{c} ) \left( \sqrt{c} \tanh \frac{r}{\sqrt{c}} 
 \frac{\partial}{\partial z}
+\left(\frac{1}{\sqrt{c}\tanh \frac{r}{\sqrt{c}}} - \frac{\tanh \frac{r}{\sqrt{c}}}{\sqrt{c} }\right) 
 \frac{\partial}{\partial \theta}\right),
\]
\[
\tZ^c=\frac{\partial}{\partial z}.
\]
Thus the dilated sublaplacian reads
\begin{align*}
\tL^c& =\frac{1}{c} e^{-\frac{1}{2} \ln c D} \tL e^{ \frac{1}{2} \ln c D} \\
  &=(\tX^c)^2 + (\tY^c)^2 \\
  & =\frac{\partial^2}{\partial r^2}+\frac{2}{\sqrt{c}} \mathrm{ cotanh } \frac{2r}{\sqrt{c}}\frac{\partial}{\partial r}+\frac{1}{c} \left(\frac{1}{\tanh \frac{r}{\sqrt{c}}}
-\tanh \frac{r}{\sqrt{c}}\right)^2   \frac{\partial^2}{\partial \theta^2} +c\tanh^2 \frac{r}{\sqrt{c}} \frac{\partial^2}{\partial z^2}+2(1-\tanh^2 \frac{2r}{\sqrt{c}})
\frac{\partial^2}{\partial z \partial \theta}.
\end{align*}

Note that the above map is well defined on $\SL$ for functions 
whose supports are included in the box $[0,\infty)\times [0,2\pi]\times [-\pi,\pi]$. 
So the dilated vector fields $X^c,Y^c$ and $Z^c$ of the vector fields $X,Y$ and $Z$ on $\SL$  are well-defined on the box $[0,\infty)\times [0,2\pi]\times [-\sqrt c \pi, \sqrt c \pi]$. 
Consequently, if $f: \mathbb{H} \rightarrow \mathbb{R}$ is a smooth function with compact support, we can speak of $X^cf$, $Y^c f$, and $Z^c f$ as 
soon as the dilation factor  $c$ is big enough. 

With these notations, the \textit{operator} analogue of the convergence of dilated $\SL$ and $\tSL$ to $\mathbb{H}$ is the following:

\begin{proposition}
If $f:\mathbb{H} \rightarrow \mathbb{R}$ is a smooth function with compact support, then, uniformly, for $A=X,Y,Z$

$$
\lim_{c \to +\infty} A^c f= \lim_{c \to +\infty} \tA^c f= \hA f 
$$
 
 and 
 $$
 \lim_{c \to \infty} L^c f=\lim_{c \to +\infty} \tL^c f=\hL f.
 $$
\end{proposition}

As a corollary, we obtain the following:
\begin{corollary}
Uniformly on compact sets of $\mathbb{R}_{\ge 0} \times \mathbb{R}$,
\[
\lim_{t \to 0} \frac{d(\sqrt t r, t z)}{\sqrt t}= \lim_{t \to 0} \frac{\td(\sqrt t r, t z)}{\sqrt t} = d_\HH (r,z)
\]
where $d_\HH$ is the Carnot-Carath\'eodory distance of the point $(r,\theta, z)$ to the origin in $\HH$.
\end{corollary}

Now we can prove the stronger result for the diffusions.
\begin{proposition}
Uniformly on compact sets of $\mathbb{R}_{\ge 0} \times \mathbb{R}$,
\[
\lim_{t \to 0} t^2 p_t (\sqrt{t} r, t z) =\lim_{t \to 0} t^2 \tp_t (\sqrt{t} r, t z)= h_1(r,z)
\]
\end{proposition}

The computations to prove this result are based on the explicit formula for $p_t$ and are very closed from the ones done in \cite{Baudoin-Bonnefont}  on the group $\SU$ for the same result, and therefore the proof will be omit.

\begin{remark}
Recall that (see \cite{bakry-baudoin-bonnefont-chafai}), due to the dilation structure on the Heisenberg group, the two following facts hold true: for all $r>0$ and all $z\in \R$,
$$
\frac{d_\HH(\sqrt t r, t z)}{\sqrt t} = d_\HH (r,z)
$$
and 
$$
t^2 h_t (\sqrt{t} r, t z) = h_1(r,z).
$$
\end{remark} 
 
 \section{Some functional inequalities for the heat kernel}
In this section we will obtain some functional inequalities and in particular some gradient bounds for the heat kernels on both $\SL$ and $\tSL$. All these bounds will be obtained only through informations of local nature. Indeed we will only use the bracket relations of the vector fields and the explicit expressions of the vector fields. All the results this section can be applied indifferently to $\SL$ or $\tSL$. Therefore, in the sequel,  $\G$ will denote equally one of the  two groups $\SL$ or $\tSL$. Moreover, by an abuse of notations, when it is not necessary to distiguish among them, we will write the corresponding vector fields and operators on $\SL$ and $\tSL$ both in the same way (we choose here to keep the simplest notation of $\SL$).

Let us recall that
\[
L=X^2+Y^2
\]
with
\[
[X, Y]=2Z, \quad [X,Z]=2Y, \quad [Y,Z]=-2X.
\]

Since we will use it a lot in the sequel we  introduce the following notations (see \cite{livre}, \cite{bakry-st-flour}). For $f$ and $g$ smooth functions on $\G$, let 
$$\Gamma(f,g) =\frac{1}{2}(L(fg)-fLg-gLf) $$ and
$$\Gamma_{2}(f,g) = \frac{1}{2}(L\Gamma(f,g) -\Gamma(f,Lg)- \Gamma(g, Lf)).$$

In the present setting, we obtain 
\[\Gamma(f,f)=X^2+Y^2
\]
and
\begin{align*}\label{Gamma2}
\Gamma_2 (f,f)= (X^2f)^2+(Y^2f)^2+\frac{1}{2} \left( (XY+YX)f \right)^2+ 2(Zf)^2 -
4 \Gamma (f,f)-4 (Xf)(YZf)+4(Yf)(XZf).
\end{align*}

\subsection {$\Gamma_2$ radial}
In this section, we will express the $\Gamma$ and the $\Gamma_2$ of a   smooth radial  function $f$ (i.e. that only depends on the variables $r$ and $z$).
\begin{align*}
\Gamma (f,f) 
 & = \left( \frac{\partial f}{\partial r}\right)^2 +\tanh^2 r \left( \frac{\partial f}{\partial z}\right)^2 ,
\end{align*}
and 
\begin{eqnarray*}
\Gamma_2 (f,f) 
 & =  & \left( \frac{\partial^2 f}{\partial r^2}\right)^2 + \left(  \frac{2}{\sinh 2 r}  \frac{\partial f}{\partial r} -\tanh^2 r \frac{\partial^2 f}{\partial z^2} \right)^2 + 2 \left( \frac{1}{\cosh^2 r} \frac{\partial f}{\partial z} +\tanh r \frac{\partial^2 f}{\partial r \partial z} \right)^2.
\end{eqnarray*}

Thus, we obtain that for a smooth radial function $f$, $\Gamma_2(f,f)\geq 0$. This is an interesting fact which may be surprising if we think that this subelliptic $\SL$ (or better $\tSL$) is the subelliptic model space with negative curvature.
\subsection{A first gradient bound}

\begin{proposition}
Let $f: \G \rightarrow \mathbb{R}$ be a smooth function. For $t >0$ and $g \in \G$,
\[
\Gamma (P_t f , P_t f)(g) \le A(t)
\left( \int_{\G} f^2 d\mu - \left( \int_{\G} f d\mu\right)^2 \right)
\]
where
\[
A(t)=-\frac{1}{4} \frac{\partial}{\partial t} \int_{\G} p^2_t d\mu.
\]
Moreover the constant $A(t)$ is decreasing.
\end{proposition}
\begin{proof}
The proof is exactly the same as the one on the $\SU$ group and on the Heisenberg group (see \cite{bakry-baudoin-bonnefont-chafai} and \cite{Baudoin-Bonnefont}). To show that $A$ is decreasing, see that:
\[
A'(t)=- \int_{\SL} \Gamma_2  ( p_t,  p_t) d\mu.
\]
Therefore, $p_t$ only depends on $(r,z)$, $\Gamma_2 ( p_t, p_t) \ge 0$ and  $A'(t) \le 0$.

\
\end{proof}
\begin{remark}
Due to the use of the Cauchy-Schwarz inequality in the previous proof, we see that the previous inequality is sharp.
\end{remark}
We now study the constant $A(t)$. Here the value of the constant $A(t)$ does depend on the choice of the space $\SL$ or $\tSL$,  but again, its  behaviour in small time is the same for both space.
\begin{proposition}
We have the following properties:
On $\tSL$,
\begin{itemize}
\item $A(t) \sim_{t \rightarrow 0} \frac{1}{256 t^3}$;
\item $A(t) \sim_{t \rightarrow +\infty} \frac{ e^{-2t}}{256 t^2}.$
\end{itemize}
 On $\SL$,
\begin{itemize}
\item $A(t) \sim_{t \rightarrow 0} \frac{1}{256 t^3}$;
\end{itemize}
\end{proposition}

\begin{proof}
We begin by $\tSL$.
We can  observe that, due to the semigroup property,
\[
\int_{\tSL} \tp^2_t d\mu=\tp_{2t} (0)
\]
and
$$
\tp_{t}(0,0)= \frac{e^{-t}}{(4\pi t)^2}  \int_{-\infty}^{+\infty} \frac{y}{\sh y} dy = \frac{e^{-t}}{32 t^2}.$$

Now we turn to $\SL$. As before the asymptotics in small times are the same as the one on $\tSL$.

\end{proof}

\subsection{Li-Yau type inequality}

We now provide a Li-Yau type estimate for the heat semigroup. This inequality appears in \cite{bbbq} but all its consequences do not appear in this paper. The idea of its proof is close to the one done in \cite{Bakry-Ledoux-LY} for elliptic operators. Let us recall it:
\begin{proposition}\label{Li-Yau}
For all $\alpha>2$, for every positive smooth function $f: \G \to \R$ and every $t >0$,
\begin{equation}\label{eqLi-Yau}
\Gamma (\ln P_t f) +\frac{ 4t}{\alpha}  (Z \ln P_t f)^2  \leq \left(
\frac{3 \alpha -1}{\alpha -1} +\frac{  t}{ 2 \alpha} \right)\frac{LP_t f}{P_t f} 
+ \frac{ 16 t}{\alpha}  + \frac{4(3\alpha-1)}{\alpha-1} + \frac{(3\alpha-1)^2}{4(\alpha-2)}
\frac{1}{t}.
\end{equation}
\end{proposition}

Let us denote 
$$A(t)=\frac{3 \alpha -1}{\alpha -1} +\frac{  t}{ 2 \alpha}$$
and 
$$B(t)=\frac{ 16 t}{\alpha}  + \frac{4(3\alpha-1)}{\alpha-1} + \frac{(3\alpha-1)^2}{4(\alpha-2)}
\frac{1}{t}.$$
$A(t)$ and $B(t)$ here are always non negative.
For $t$ small, $A(t)$ is of the order of a constant and $B(t)$ is of order of $\frac{C}{t}$.

For, $t$ big, one can choose $\alpha=t$ and get both $A(t)$ and $B(t)$ of the order of a constant.      

\begin{remark}
It can be shown that with this choice $\alpha=t$, the constants $A(t)$ and $B(t)$ are of the best order possible in the differential system that appears in \cite{bbbq}.  
\end{remark}
As a direct corollary of the Li-Yau type inequality of proposition \ref{Li-Yau}, we classically deduce (by integrating along 
geodesics) the following Harnack type inequalities: 
\begin{proposition}
There exist two positive constant $A_1$ and $A_2$ such that for $0<t_1<t_2<1$ and $g_1,g_2 \in \G$
\begin{align}\label{Harnack}
\frac{p_{t_1} (g_1)}{p_{t_2} (g_2)} \le \left( \frac{t_2}{t_1} \right)^{A_1} 
\exp\left(A_2 \frac{d (g_1,g_2)^2}{t_2-t_1}\right)
\end{align}

and there exist two positive constants $\tilde A_1$ and $\tilde A_2$ such that for $2<t_1<t_2$ and $g_1,g_2 \in \SL$
\begin{align}\label{Harnack2}
\frac{p_{t_1} (g_1)}{p_{t_2} (g_2)} \le \exp\big( \tilde A_1 (t_2-t_1)\big)
\exp\left(\tilde A_2 \frac{d (g_1,g_2)^2}{t_2-t_1}\right)
\end{align}
where $d (g_1,g_2)$ denotes the Carnot-Caratheodory distance from $g_1$ to $g_2$.
\end{proposition}

As another corollary of the Li-Yau inequality, we can also prove the following global estimate:

\begin{proposition}\label{racine_gamma_log}
There exists a constant $C >0$ such that for  $t \in (0,1)$, $r >0$, $z \in [-\pi,\pi]$,
\[
\sqrt{\Gamma (\ln p_t)(r,z) } \le C \left( \frac{d(r,z)}{t} +\frac{1}{\sqrt{t}} \right),
\]
and there exists a constant $\tilde C >0$ such that for  $t >2$, $r >0$, $z \in [-\pi,\pi]$,
\[
\sqrt{\Gamma (\ln p_t)(r,z) } \le \tilde C \left( \frac{d(r,z)}{t} +1 \right),
\]
\end{proposition}

\begin{proof}
The proof is the same as on $\SU$ (see \cite{Baudoin-Bonnefont}) since it is only based on the preceeding Harnack inequalities and the positivity of the $\Gamma_2$ of a radial function. The only difference is that in the second point we have to use the Harnack inequality (\ref{Harnack2}) in big times.
\end{proof}

\subsection{The reverse spectral gap inequality}

As  in the Heisenberg group case and in the $\SU$ case  (see \cite{bakry-baudoin-bonnefont-chafai} and \cite{Baudoin-Bonnefont}), we can easily obtain a reverse Poincare inequality
with a sharp constant for the subelliptic heat kernel measure on $\SL$ and on $\tSL$.

\begin{proposition}
Let $f: \G \rightarrow \mathbb{R}$ be a smooth function. For $t >0$ and $g \in \G$,
\[
\Gamma (P_t f , P_t f)(g) \le C(t)
\left( P_t f^2 (g) - (P_t f)^2 (g) \right)
\]
where
\[
C(t)=-\frac{1}{2} \frac{\partial}{\partial t} \int_{\G} p_t  \ln p_t d\mu.
\]
Moreover, this constant $C(t)$ is decreasing.
\end{proposition}

\begin{proof}
As before, the proof 
is exactly the same as on Heisenberg and  on the $\SU$ group (see \cite{bakry-baudoin-bonnefont-chafai} and \cite{Baudoin-Bonnefont}).To see that $C(t)$ is decreasing, note that, after some computations,:
\[
C'(t)=- \int_{\G} \Gamma_2  (\ln p_t, \ln p_t) p_t d\mu.
\] 
But as before, let us observe that $p_t$ only depends on $(r,z)$, thus $\Gamma_2 (\ln p_t,\ln p_t) \ge 0$ and $C'(t) \le 0$.
\end{proof}

\begin{remark}
Due to the use of the Cauchy-Schwarz inequality in the previous proof, we see that the previous inequality is sharp.
\end{remark}

We now study the constant
\[
C(t)=-\frac{1}{2} \frac{\partial}{\partial t} \int_{\G} p_t  \ln p_t d\mu.
\]
Let us recall:
\[
\int_{\G} \frac{ \Gamma (p_t,p_t) }{p_t} d\mu=\int_{\G} \Gamma (\ln p_t,\ln p_t) p_t d\mu
=-\int_{\G} \ln p_t Lp_t   d\mu=- \frac{\partial}{\partial t} \int_{\G} p_t  \ln p_t d\mu.
\]
This constant does depend of course of the choice of the space $\SL$ or $\tSL$ but has the same behaviour when $t$ is small.
\begin{proposition}
On both $\tSL$ and $\SL$,
\begin{itemize}
\item $C(t)\sim_{t \rightarrow 0} \frac{1}{t}$;
\end{itemize}
\end{proposition}

\begin{proof}

We work here with $\tSL$.
We now study $C(t)$ when $t \to 0$. The idea is that, asymptotically when $t \to 0$, the constant $C(t)$ has to behave like 
the best constant of the reverse spectral gap inequality on the Heisenberg group (see the Section \ref{convH}). From 
\cite{bakry-baudoin-bonnefont-chafai}, this constant is known to be $1/t$. This is also the case for $\SU$ (see \cite{Baudoin-Bonnefont}). 
Let $0<t<1$ we have:
\begin{align*}
tC(t)& =\frac{t}{2} \int_{\tSL} \tp_t \tGa (\ln \tp_t,\ln \tp_t)  d\mu \\
 &=\int_{r>0} \int_{z=-\infty}^{\infty} t^{5/2} \frac{\sinh 2\sqrt{t} r}{2}
\tp_t(\sqrt{t} r,tz)  \tGa (\ln \tp_t,\ln \tp_t) (\sqrt{t}r,tz) dr dz
\end{align*}
Now, by using the result of Section \ref{convH}, we easily obtain that, the following pointwise convergences hold
\[
\lim_{t \to 0} t^{3/2} \frac{\sinh 2\sqrt{t} r}{2}
\tp_t(\sqrt{t} r,tz)= h_1 (r,z) r  
\]
\[
\lim_{t \to 0} t \tGa (\ln \tp_t , \ln \tp_t )( \sqrt t r, t z)=\hGa (\ln h_1 ) (r,z),
\]
where $h_t (r,z)$ and  $\hGa$ are defined in Section \ref{convH} (see \ref{gaveau}, \ref{Xheisenberg}, \ref{Yheisenberg}).


Thanks to Proposition \ref{racine_gamma_log}, there exists a constant $C>0$ such that
\[
t \tGa (\ln \tp_t , \ln \tp_t )(\sqrt t r,t z) \le C\left(1+ \frac{\td(\sqrt t r, t z)}{\sqrt t}\right)^2, \quad t \in (0,1).
\]
 and thanks to Proposition \ref{upper_estimate}, there exist two constant $C_1, C_2 >0$ such that
\[
t^2 \tp_t( \sqrt t r, t z) \leq C_1 \exp \left(- C_2 \frac{\td^2(\sqrt t r , tz)}{t} \right).  
\]
Also we have:
\[
\frac{\sh 2 \sqrt t r}{\sqrt t}\leq e^{2r}.
\]
Eventually, by the estimates of the distance of Proposition \ref{estimates_distance}, the dominated convergence theorem implies
\[
\lim_{t \to 0} t C(t)=\frac{1}{2} \int_{\mathbb{R}^3} h_1 (r,z) 
\hGa (\ln h_1 ) (r,z)  r dr d\theta dz.
\]
This last expression is equal to 1, according to \cite{bakry-baudoin-bonnefont-chafai}.

 \end{proof}
 
\begin{remark}
We can ask about the behaviour of $C(t)$ on $\SL$ and $\tSL$ as $t$ goes to infinity. 
By using proposition \ref{Li-Yau} and its notation, for a positive function $f$,
\[
\int P_t (f) \Gamma(\ln P_t f) d\mu \leq B(t) \int f d\mu.
\]
By taking $f$ an approximation of the unity, we obtain:
\[
C(t) \leq B(t).
\]
And so for big $t$, $C(t)$ is less than a constant we can compute.
\end{remark}

\subsection{Some isoperimetrics inequalities}
We can now recover some isoperimetric results from the Li-Yau inequality. We use methods of Varopoulos and Ledoux (see \cite{Varopoulos} and \cite{Ledoux_StFlour}). As before, $\G$ denotes indifferently $\SL$ or $\tSL$.
   
First we set:

\begin{proposition}
There exists $C$ such that for every smooth function $f$ on $\G$ and every $0<t<1$,
$$|| \sqrt{\Gamma  P_t f} ||_\infty \leq \frac{C}{\sqrt t} ||f||_\infty.
$$
There exists $C'$ such for every smooth function $f$ on $\G$ and $0<t<1$,
$$||  f - P_t f||_1 \leq C' \sqrt t || \sqrt{\Gamma f} ||_1.
$$
\end{proposition}

\begin{proof}
Indeed, for the first point, the Li-Yau inequality gives for $0<t<1$ and $f$ a positive function:
$$L(P_t f)^- \leq \frac{C}{t} P_t f. 
$$
By integrating against $\mu$ and noticing $\int L(P_t f) d\mu= 0$, we get
$$\frac{1}{2} \int |L(P_t f )|d\mu \leq \frac{C}{t} \int  f d \mu.
 $$
 Then $||L P_t f||_1 \leq \frac{2 C}{t} ||f||_1$, and since $L P_t$ is self-adjoint, by duality
 $||L P_t f||_\infty \leq \frac{2 C}{t} ||f||_ \infty.$
 By pluging-in this result in the Li-Yau inequality (\ref{Li-Yau}), 
 $$\Gamma P_t f \leq \frac{C'}{t} ||f||_\infty P_t f$$
 which implies the  first result.
 
 For the second point, let $f$ and $g$ be two smooth functions, 
 \begin{eqnarray*}
 \int g (P_t f-f) d\mu &=& \int_0^t \int g L P_s f d\mu ds\\
                       &=&  -\int_0^t \int  \Gamma( P_s g, f) d\mu ds\\
 \end{eqnarray*}
 Since $\Gamma(P_s g,f) \leq \sqrt{\Gamma P_s g} \sqrt{\Gamma f}$, by the first point, we have
 \begin{eqnarray*}
 |\int g (P_t f-f) d\mu| &\leq& C||g||_\infty \int_0^t \frac{1}{\sqrt s} ds \int \sqrt{\Gamma f} d\mu\\
                         &=& 2 C \sqrt t ||g||_\infty \int \sqrt{\Gamma f} d\mu.
 \end{eqnarray*}
 By letting $g$ tend to $sign(P_t f-f)$, we end the proof. 
\end{proof}

And actually these last results will enable us to obtain some isoperimetric inequalities on small sets. For $A$ and $B$ measurable sets, let us denote 
$$K_t(A,B) = \int_B P_t(1_A) d\mu.
$$
It is easy to see that 
$$K_t(A,A^c)= \mu(A) -K_t(A,A)
$$
and 
$$ K_t(A,A)= ||P_\frac{t}{2} 1_A||_2 ^2.
$$
 We have the following proposition:
\begin{proposition}\label{isop}
Let $A$ be  a measurable set of $\G$ which is a Cacciopoli set and  call $P(A)$ its perimeter  (see \cite{Garofalo-Nhieu} and the references therein to see their definition in our context) then   
\begin{equation}K_t(A,A^c) \leq 2 C \sqrt t P(A).
\end{equation}

Now assume also $\mu(A)$ is small enough, 
then 
$$ \mu(A)^\frac{Q-1}{Q} \leq C P(A)
$$ 
for some positive constant C and $Q=4$ stands for the homegenous dimension of the group.
\end{proposition}

\begin{proof}
Let $A$ be  a measurable set of $\G$ and let $f$ and $g$ be two smooth functions which aproximate respectively $1_{A}$ and $1_{A^c}$ and with  $||g||_\infty \leq 1$. Then the quantity $\int g (P_t f-f) d\mu$ converges towards $K_t(A,A^c)$ and as before
\begin{eqnarray*}
\int g (P_t f-f) d\mu &\leq& ||g||_\infty ||P_t f - f||_1\\
                      &\leq & 2 C \sqrt t \int \sqrt{\Gamma f} d\mu  
\end{eqnarray*}
As it is well known, we can choose $f$ such that $\int \sqrt{\Gamma f} d\mu$ tends towards $P(A)$ (see theorem 1.14 of \cite{Garofalo-Nhieu}), so we obtain 
$$K_t(A,A^c) \leq 2 C \sqrt t P(A).
$$
Therefore, 
$$ P(A)\geq \frac {C'}{\sqrt t} (\mu(A)-||P_\frac{t}{2} 1_A||_2 ^2).$$
Using the ultracontractivity in small times, we get $||P_t f||_\infty \leq \frac{C}{t^{Q/2}} ||f||_1$ and by interpolation  $||P_t f||_2 \leq \frac{\sqrt C}{t^{Q/4}} ||f||_1$, so 
$$ P(A)\geq \frac {C'}{\sqrt t} \mu(A)\left(1-\frac{C}{\left(\frac{t}{2}\right)^{Q/2}}\mu(A)\right).$$
Now we will have to optimize the function of $t$ on the right-hand side. We see this function attains a positive maximum for $t$ of the order $\mu(A)^{\frac{2}{Q}}$ which has value of order $\mu(A)^\frac{Q-1}{Q}$.  
\end{proof}

\begin{remark}
In all our previous results, we can give an explicit bound on the constants that appeared from the constant that appeared in the Li-Yau inequality \ref{eqLi-Yau}. 
\end{remark} 
\begin{remark}
It is known that the result of Proposition \ref{isop} is true for all sets (see Theorem 7.5 of \cite{Chanillo-Yang} and note that the space $\tSL$ has constant curvature $R=-1$). It seems that the Proposition \ref{Li-Yau} is far from being optimal in big times. 
\end{remark} 

\section*{Acknowledgements}
The author would like to thank both the referee and Laurent Saloff-Coste for pointing and explaining an important mistake in a previous version of this article.

\end{document}